\documentclass[article,12pt]{article}
\usepackage{amsmath,amssymb}
\usepackage{amsthm}
\usepackage{mathtools}
\usepackage{graphics,epsfig,color}
\usepackage[mathscr]{euscript}
\usepackage[notcite,notref]{}
\usepackage{thmtools}

\usepackage{comment}
\usepackage{enumerate}

\usepackage{graphicx}
\usepackage{graphics}
\usepackage{caption}
\usepackage{subfig}
\usepackage{hyperref}
\usepackage{tikz}
\usetikzlibrary{arrows}

\usepackage{tabularx}
\usepackage{bm}
\usepackage{multirow}
\usepackage{pifont}

\usepackage{ bbold }

\theoremstyle{plain}
\newtheorem{theorem}{Theorem}[section]

\newtheorem{lemma}[theorem]{Lemma}
\newtheorem{proposition}[theorem]{Proposition}

\theoremstyle{definition}

\theoremstyle{remark}
\newtheorem{remark}{Remark}

\makeatletter
\def\ps@pprintTitle{ 
	\let\@oddhead\@empty
	\let\@evenhead\@empty
	\def\@oddfoot{\footnotesize\itshape
		\ifx\@empty\@empty
		\else\@journal\fi\hfill\today}%
	\let\@evenfoot\@oddfoot
}

\newcommand{\kws}{\bigskip\par\addvspace\medskipamount{\rightskip=0pt plus1cm}{\noindent \bfseries Keywords: \enspace}}

\newcommand{\MSC}{\bigskip\par\addvspace\medskipamount{\rightskip=0pt plus1cm}{\noindent \bfseries Mathematics Subject Classification (2000): \enspace}}

\numberwithin{equation}{section}
\numberwithin{theorem}{section}

\newcommand{\R}{\mathbb R}
\newcommand{\N}{\mathbb N}
\newcommand{\Id}{\boldsymbol{\mathrm{Id}}}

\usepackage{natbib}

\usepackage{authblk}

\title{Reinforced dynamics for interacting agents in competitive or cooperative environments} 

\author[1]{	Michele~Aleandri\thanks{Email: \textit{maleandri@luiss.it}}}
\author[2]{Paolo~Dai Pra\thanks{Email: \textit{paolo.daipra@univr.it}}}
\author[3]{Ida Germana~Minelli\thanks{Email: \textit{idagermana.minelli@univaq.it}}}

\affil[1]{{\small LUISS University, DIASDD, 00197 viale Romania 32, Roma,  Italy.}}
\affil[2]{{\small Universit\`a di Verona, 37134 Strada le Grazie 15, Verona, Italy.}}
\affil[3]{{\small Universit\`a dell'Aquila, DISIM, 
		67100 Coppito, L'Aquila, Italy.}}
\date{}

\begin{document}
	
	\maketitle

	\begin{abstract}We study systems of interacting reinforced stochastic processes, where agents’ decisions evolve under reinforcement, network-mediated interactions, and environmental influences. In competitive environments with irreducible networks, we prove almost sure convergence of agents’ states, with bipartite graphs leading to non-deterministic limits and non-bipartite graphs to deterministic ones. These results are extended to reducible networks with cooperative or competitive substructures, using a hierarchical graph representation in which higher-level agents perceive lower-level agents as forcing inputs. A general framework for interacting agents under external forcing is also provided, offering a rigorous characterization of long-term behavior and synchronization phenomena in such systems.
	\end{abstract}

	\kws Interacting systems; Reinforced stochastic processes; Network-based dynamics; Urn models; Opinion dynamics. 
	
	\MSC 60F15; 60K35; 62L20; 91D30. 

	\section{Introduction}
		Interacting agents make decisions and form preferences or beliefs under the influence of others through various forms of social interaction. For example, opinions are often shaped by social influence, leading individuals to either conform to or diverge from the majority. A particularly interesting mechanism in these dynamics is reinforcement, where the probability of an action or event occurring increases the more frequently it has occurred in the past. In the context of opinion dynamics, self-reinforcement can be interpreted as a process whereby an agent’s personal inclination is strengthened in the direction of a previously expressed choice or action. Alongside this force, the interaction is described by the weighted network that gives a measure of how strong the link between two agents is. Moreover, these interactions can be modulated by the agents’ environment: in cooperative settings, an agent tends to align with their neighbors’ opinions, whereas in competitive settings, an agent tends to act in opposition.
		The interplay between reinforcement and interaction has been the leading subject of a wide literature \cite{aletti2017synchronization,Aletti_Crimaldi_Ghiglietti_2019,Aletti_Crimaldi_Ghiglietti_2020,Aletti_Ghiglietti_2017,aletti2024networksProb,aletti2024networks,Benaim_Benjamini_Chen_Lima_2015,Chen_Lucas_2014,Crimaldi_DaiPra_Louis_Minelli_2019,Crimaldi2016,Crimaldi2023,dai2014synchronization,Lima2016,Louis2018,mirebrahimi2024synchronization,Paganoni2004,Sahasrabudhe2016}. While the case of cooperative interaction is well understood, the competitive case has never been treated in details. We remark that, in absence of reinforcement, competitive interaction has been widely treated in the Statistical Mechanics modeling of antiferromagnets (see e.g \cite{binek2003ising,wu1989critical}) and it has been applied in social science settings to describe nonconformist behavior \cite{AleMin19Lotka,AleMin21delay,touboul2019hipster}.
		
		In this paper we analyse the long-term behavior of a system of competitive interacting agents subject to reinforcement.
		In more detail, in Section \ref{sec:irred} we prove the almost sure convergence of the agents’ states when the interaction network is represented by an irreducible graph and the agents are in a competitive environment. Notably, two possible scenarios arise depending on whether the graph is bipartite (leading to a non-deterministic limit) or not (leading to a deterministic limit). In Section \ref{sec:reducble_graph}, we extend the convergence result to reducible graphs, where the environment—restricted to agents of the same family—can be either competitive or cooperative. This result is established using a hierarchical representation of the graph, in which agents are divided into classes of different levels, and higher-level agents perceive lower-level agents as a forcing input. A general treatment of interacting agents with forcing inputs is presented in Section \ref{sec:competitiveforce}.

		\section{Convergence and synchronization for irreducible graphs}\label{sec:irred}
		
		We consider the discrete-time dynamics of $N$ interacting agents whose states $Z_n(i) \in [0,1]$, $n \geq 0$, $i \in \{1,2, \ldots,N\}$ and driving noise $Y_n(i) \in \{0,1\}$ evolve as follows:
		\begin{align} \label{equation1}
			&Z_{n+1}(i)=(1-r_n)Z_n(i) + r_n Y_{n+1}(i),\ \ \ &i=1,\ldots, N,\\
			&P\left(Y_{n+1}(i)=1|\mathcal{F}_n\right) = \alpha_{ii}Z_n(i)+ \sum_{j\neq i}\alpha_{ij} \left(1-Z_n(j)\right),\ \ \  &i=1,\ldots, N,\nonumber
		\end{align} 
		where $\mathcal{F}_n=\sigma(Y_k; \ k\leq n)$ and
		\begin{itemize}
			\item 
			$(r_n)_{n \geq 0}$ is a sequence in $(0,1)$ satisfying the conditions
			\[
			\sum_{n=0}^{+\infty} r_n = +\infty, \quad \ \ \sum_{n=0}^{+\infty} r^2_n < +\infty.
			\]
			\item
			$\mathbf{A}=(\alpha_{i j})_{ i,j=1,\ldots, N}$ is a stochastic matrix tuning the interaction between agents. Note that the term $\alpha_{ij}(1-Z_n(j))$ 
			describes a {\em competitive} interaction: the smaller the value of $Z_n(j)$ the higher its contribution in the probability of increasing the value of $Z_n(i)$. 
			
		\end{itemize}
		Dynamics of this form emerge naturally in {\em urn models}. For instance, consider $N$ urns with red and white balls. At each time step a ball is randomly drawn from each urn, then reinserted, and a further ball is added according to the following rule: the ball added in urn $i$ is with probability $\alpha_{ii}$ of the same color of that drawn from urn $i$, with probability $\alpha_{ij}$, $j \neq i$, of the color different from that drawn form urn $j$. If we denote by $Z_n(i)\in [0,1]$ and $Y_n(i)\in
		\{0, 1\}$  respectively the random variables representing the fraction of red balls in urn $i$ at time $n$ and the indicator function of the event \emph{"the ball added in urn $i$ at time $n$ is red"}, their evolution is given by \eqref{equation1} with $r_n = \frac{1}{m+n+1}$, where $m$ is the initial number of balls in the urn.

		Going back to the general dynamics \eqref{equation1}, 
		we can write the evolution of the variables $Z_n(i)$ in vector form as 
		\begin{equation}\label{equation2}
			\mathbf{Z}_{n+1}=\mathbf{Z}_n+ r_n (\mathbf{K}\mathbf{Z}_n+\mathbf{c})+ r_n \mathbf{\Delta M}_{n+1},
		\end{equation}
		where $\mathbf{Z}_n=(Z_n(i))_{i=1}^N,\ \mathbf{Y}_{n}=(Y_{n}(i))_{i=1}^N$, $\mathbf{K}=(k_{ij})_{i,j=1}^N$ with $k_{ii}=\alpha_{ii}-1$ for $i=1.\ldots, N$ and $k_{i, j}=-\alpha_{ij}$ for $i,j=1,\ldots, N,\ i\neq j$,\  
		$ \mathbf{\Delta M}_{n+1}=(\Delta M_{n+1}(i))_{i=1}^N$ with $\Delta M_{n+1}(i)= Y_{n+1}(i)-E[Y_{n+1}(i)|\mathcal{F}_n]$, and $\mathbf{c}=(c_i)_{i=1}^N$ with $c_i=1-\alpha_{ii}$ for $i=1,\ldots, N$. \\

		In what follows, for a given $N$-dimensional matrix $\mathbf{A}=(\alpha_{ij})_{i,j=1}^N$ we shall denote by $\mathrm{diag}(\mathbf{A})=(d_{ij})_{i,j=1}^N$ the diagonal matrix with $d_{ii}=\alpha_{ii}$, for $i=1,\ldots, N$. We use the notation $\mathbf{I}_N$ for the $N$-dimensional identity matrix. Then, the matrix $\mathbf{K}$ in equation \eqref{equation2} can be written as $\mathbf{K}=2\mathrm{diag}(\mathbf{A})-\mathbf{I}_N-\mathbf{A}$.

		We can associate a weighted oriented graph $G=(V,E, W)$ to the system of agents, where the vertices represent the agents, we have an oriented edge $(i,j)\in E$ if and only if $i\neq j$ and $\alpha_{i, j}\neq 0$ and $W(i,j)=\alpha_{ij}$, for any $(i,j)\in E$. \\
		We say that agent $i$ is \emph{accessible} by agent $j$ if either $(i,j)\in E$ or there exists an oriented path from $i$ to $j$, i.e., if there exists $n$ and a sequence $j_1,\ldots, j_n\in \{1,\ldots N\}\setminus\{i, j\}$ of distinct vertices such that 
		$\alpha_{ij_1}\alpha_{j_1j_2}\cdot\ldots\cdot\alpha_{j_nj}>0$. Two agents $i$ and $j$ \emph{communicate} if $i$ is accessible by $j$ and $j$ is accessible by $i$. We say that $G$ is \emph{irreducible} if every pair of agents communicate. 
		
		The main result of this section is the following 
		
		\begin{theorem}[Convergence and Synchronization]\label{convergence}
			If the graph $G$ is irreducible, then $(\mathbf{Z}_n)_n$ converges almost surely when $n\to \infty$ to a vector $\mathbf{Z}_\infty=(Z_\infty(i))_{i=1}^N$. 
			When the graph $G$ is bipartite, 
			letting $I, J$ be the sets of vertices of the partition, the limit $\mathbf{Z}_\infty$ has the form $Z_\infty(i)=Z_\infty$ for $i\in I$ and $Z_{\infty}(i)=1-Z_\infty$ for $i\in J$, where $P(Z_\infty=x)<1$ for every $x\in [0,1]$. When the graph is not bipartite the limit is deterministic and is given by $\frac{1}{2}\mathbf{1}$, where $\mathbf{1}$ denotes the $N$-dimensional vector of ones.
		\end{theorem}
		
		\subsection{Preliminary results}
		
		In order to prove Theorem \ref{convergence} we need some preliminary results. 
		\begin{lemma}\label{Lemma1}
			Let $(X_n)_n$ be an irreducible Markov chain with values in $V=\{1,2,\ldots, N\}$ and  transition matrix $\mathbf{Q}$. Assume that $q_{ii} = 0$ for all $i \in V$. Let $G$ be the weighted oriented graph associated to $\mathbf{Q}$. Then the chain has even period if and only if the graph $G$ is bipartite. 
		\end{lemma}
		\begin{proof}
			Assume $G$ is bipartite: so $V = V_1 \cup V_2$, where $V_1$ and $V_2$ are nonempty, $V_1 \cap V_2 = \emptyset$, and $q_{ij}>0$ implies that either $i \in V_1$ and $j \in V_2$ or $i \in V_2$ and $j \in V_1$. Indicating with $q_{ij}(n)$ the elements of the $n$-steps transition probability, it follows immediately that $q_{ii}(n) > 0$ implies that $n$ is even, so the chain has even period. 
			
			Conversely, suppose the chain has even period.
			Denote by $\N^*$ and $2\N^+$ respectively the sets of positive odd and even integers.
			Consider the sets $V_1=\{i\in V : \exists n\in \N^* \mbox{ with } q_{1i}(n)>0\}$ and 
			$V_2=\{i\in V: \exists n\in 2\N^+ \mbox{ with } q_{1i}(n)>0\}$, which are the sets of states reachable from 1 in an even (respectively, odd) number of transitions. Since the chain is irreducible, we have $V_1\cup V_2=V$. Suppose now $i \in V_1$ and $q_{ij}>0$. Then $j$ is reachable from $1$ in an even number of steps, so $j \in V_2$. Similarly, if $i \in V_1$ and $q_{ij}>0$ one shows that $j \in V_1$. To complete the proof that $G$ is bipartite, we need to show that $V_1 \cap V_2 = \emptyset$. Suppose $i \in V_1 \cap V_2$; then there exist $n \in \N^*$ and $m \in 2\N^+$ such that $q_{1i}(n)>0$ and $q_{1i}(m) >0$. Since the chain is irreducible, there exists $h \in \N$ with $q_{i1}(h)>0$, so that $q_{11}(n+h) \geq  q_{1i}(n)q_{i1}(h) >0$ and, similarly $q_{11}(m+h)>0$. Since either $n+h$ or $m+h$ is odd, this contraddicts the assumption that the period is even.
		\end{proof}      
		
		\begin{proposition}\label{matrixK}
			Suppose that the graph $G$ is irreducible, then the matrix $\mathbf{K}$ in equation \eqref{equation2} is invertible if and only if the graph $G$ is not bipartite. 
		\end{proposition}
		\begin{proof}
			Let $\mathbf{L}$ be the matrix whose elements are defined by
			\[
			l_{ij} =  \frac{k_{ij}}{k_{ii}}.
			\]
			Note that $k_{ii} = \alpha_{ii} - 1 <0$ by the irreducibility assumption on $\mathbf{A}$. Moreover, $\mathbf{L}$ has the same rank as $\mathbf{K}$, it has nonnegative entries, $G$ as associated graph, and
			\[
			\sum_{j:j \neq i} l_{ij} = \frac{1}{1-\alpha_{ii}} \sum_{j:j \neq i} \alpha_{ij} = 1.
			\]
			It follows that $\mathbf{Q} := \mathbf{L} - \Id$ is a stochastic matrix with null diagonal elements. Thus, $\mathbf{K}$ is degenerate if and only if $-1$ is a eigenvalue of $\mathbf{Q}$ which, in turn, is equivalent to $\mathbf{Q}$ having even period. By Lemma \ref{Lemma1} this is equivalent to the fact that $G$ is bipartite. This completes the proof.
		\end{proof}
		
		We now recall two known results. The first is a standard result in {\em stochastic approximation} (it is a simple consequence of Propositions 4.1 and 4.2, and Theorem 6.9 of \cite{Benaim99}).
		\begin{theorem} \label{benaim}
			In a filtered probability space $(\Omega, \mathcal{A}, \mathbb{P}, \{\mathcal{F}_n\}_{n})$ let $(\mathbf{W}_n)_{n}$ be a $\mathbb{R}^N$-valued martingale difference (i.e. $\mathbb{E}(\mathbf{W}_n|\mathcal{F}_n) = 0$ for every $n$), and let $(\mathbf{X}_n)_{n}$ be defined by

			\[
			\begin{split}
				\mathbf{X}_{n+1} & = \mathbf{X}_n + r_n\left( F(\mathbf{X}_n) + \mathbf{W}_{n+1} \right), \\
				\mathbf{X}_0 & = \mathbf{x}_0 \in \R^N,
			\end{split}
			\]
			
			where $F:\R^N \rightarrow \R^N$ is a smooth function such that the dynamical system $\dot{\mathbf{x}} = F(\mathbf{x})$ has a unique, globally stable equilibrium $\mathbf{x}^* \in \mathbb{R}^N$. Assume there exists a constant $C$ such that $\|\mathbf{X}_n\| + \| \mathbf{W}_n\| \leq C$ for every $n \geq 1$. Then $(\mathbf{X}_n)_n$ converges a.s. to $\mathbf{x}^*$ as $n \rightarrow +\infty$.
		\end{theorem}
		
		The next result is an easy consequence of Theorem 2.2 in \cite{aletti2024networks} and Theorem 3.3(i) in \cite{aletti2024networksProb}. It deals with the {\em cooperative} version of the system in \eqref{equation1}, namely
		\begin{align} \label{equation2}
			&\tilde{Z}_{n+1}(i)=(1-r_n)\tilde{Z}_n(i) + r_n \tilde{Y}_{n+1}(i), &i=1,\ldots, N,\\
			&P\left(\tilde{Y}_{n+1}(i)=1|\mathcal{F}_n\right) = \alpha_{ii}\tilde{Z}_n(i)+ \sum_{j\neq i}\alpha_{ij} \tilde{Z}_n(j),  &i=1,\ldots, N.\nonumber
		\end{align} 
		
		\begin{theorem}\label{convergenceCoop}
			The random sequences $\big(\tilde{Z}_{n}(i)\big)_{n},  \ i=1,\ldots N$,  converge to the same limit random variable $\tilde{Z}_{\infty}$, namely \, $
			\mathbf{\tilde{Z}}_n \xrightarrow{\text { a.s. }} \tilde{Z}_{\infty} \mathbf{1}
			$. Moreover, $P(\tilde{Z}_\infty=x)<1$ for every $x\in [0,1]$.
			
		\end{theorem}
		
		We finally recall a result for {\em diagonally dominant matrices}, whose proof is an immediate consequence of the first and third Ger\v{s}gorin Theorems (see Theorems 1.1 and 1.12 in \cite{varga2011gervsgorin}).
		\begin{theorem}\label{gervsgorin}
			Let $\mathbf{K}$ be a $N \times N$ real matrix with nonpositive entries.
			\begin{itemize}
				\item[(i)]
				Suppose that $\mathbf{K}$ is {\em diagonally dominant}: $|k_{ii}| \geq \sum_{j \neq i} |k_{ij}|$ for every $i =1,\ldots,N$. Then nonzero eigenvalues of $\mathbf{K}$ have strictly negative real part.
				\item[(ii)]
				Suppose that $\mathbf{K}$ is diagonally dominant, the associated graph is irreducible, and the strict inequality $|k_{ii}| > \sum_{j \neq i} |k_{ij}|$ holds for at least one $i =1,\ldots,N$. Then $\mathbf{K}$ is invertible.
			\end{itemize}
		\end{theorem}
		
		\subsection{Proof of Theorem \ref{convergence}  }
		
		We first consider the case in which the graph $G$ is not bipartite that, by Proposition \ref{matrixK}, is equivalent to the fact that $\mathbf{K}$ is invertible. To the dynamics \eqref{equation2} we can apply Theorem \ref{benaim} with $F(\mathbf{x}) = \mathbf{Kx}+\mathbf{c}$. Indeed, we first notice that $F\left(\frac{1}{2}\mathbf{1}\right) = 0$, so that 
		\[
		F(\mathbf{x}) = \mathbf{K} \left(\mathbf{x}- \frac{1}{2}\mathbf{1}\right).
		\]
		Moreover, observing that $k_{ii}=\alpha_{ii}-1=-\sum_{j\neq i}\alpha_{ij}$, by Theorem \ref{gervsgorin} we have that all the eigenvalues of $\mathbf{K}$ have negative real part. Thus the assumption of Theorem \ref{benaim} are satisfied, and therefore $(\mathbf{Z}_n)_n$ converges a.s. to $\frac{1}{2}\mathbf{1}$.
		
		Now assume that $\mathbf{K}$ is not invertible, so, by Proposition \ref{matrixK}, 
		the graph is bipartite. Denote by $I,J \subset \{1,2,\ldots,N\}$ the components of a bipartition, so $\alpha_{ij}=0$ if $i,j\in I$ or $i,j\in J$. 
		Define \\
		$$\begin{array}{llcc}
			\widetilde{Z}_n(i)=Z_n(i),\ \ \ \ \ \ \ \ \ \widetilde{Y}_n(i)=Y_{n}(i),     &\    \mbox{if }\  i\in I;\\
			\widetilde{Z}_n(i)=1-Z_n(i),\ \ \ \ \widetilde{Y}_n(i)=1-Y_n(i),     &\  \mbox{if }\ i\in J.
		\end{array}$$
		Then, for $i\in I$,  
		$\ \widetilde{Z}_{n+1}(i)=(1-r_n)Z_n(i)+r_n Y_{n+1}$ with $P(\widetilde{Y}_{n+1}=1|\mathcal{F}_n)=P(Y_{n+1}=1|\mathcal{F}_n)=\alpha_{ii}+\sum_{j\in J}\alpha_{ij}\widetilde{Z}_n(i)\widetilde{Z}_j=\sum_{j=1}^N \alpha_{ij}\widetilde{Z}_j$. 
		For $i\in J$ we have 
		\begin{align*}
			\widetilde{Z}_{n+1}(i)&=1-Z_{n+1}(i)=1-[(1-r_n)Z_n(i)+r_n Y_{n+1}(i)]=(1-r_n)[1-Z_n(i)]+r_n[1-Y_{n+1}(i)]\\
			&= (1-r_n)\widetilde{Z}_n+r_n\widetilde{Y}_{n+1},
		\end{align*} 
		with 
		\begin{align*}
			P(\widetilde{Y}_{n+1}(i)=1|\mathcal{F}_n)&=P(Y_{n+1}=0|\mathcal{F}_n)=1-[\alpha_{ii}Z_n(i)+\sum_{j\in I}\alpha_{ij}(1-Z_n(j))]\\
			&=\alpha_{ii}+\sum_{j\neq i}\alpha_{ij}-[\alpha_{ii}Z_n(i)+\sum_{j\neq i}\alpha_{ij}(1-Z_n(j))]\\
			&=\sum_{j=1}^N\alpha_{ij}[1-Z_n(j)]=\sum_{j=1}^N \alpha_{ij}\widetilde{Z}_j.   
		\end{align*} 
		We have thus obtained a system of cooperative reinforced processes, to which we can apply Theorem \ref{convergenceCoop}, completing the proof.
		\begin{remark}
			Even though fluctuation results are beyond the scope of this paper, they can be established for non-bipartite graphs by applying results from the stochastic approximation literature. In the bipartite case, if the matrix $A$ is diagonalizable, we obtain the same equations as in the cooperative model, and the results in \cite{aletti2017synchronization} can be applied. However, if $A$ is not diagonalizable, to the best of our knowledge, no results are currently available.
		\end{remark}

		\section{Extension to reducible graphs}\label{sec:reducble_graph}
		
		In this section we continue to study a system given by \eqref{equation1}, but relaxing the assumption of irreducibility of $G$. To illustrate the argument we first consider a simple model, where the graph contains two disjoint classes of communicating agents $T_1, R_1\subset \{1,\ldots, N\}$ such that there are no edges $(i,j)\in E$ with $i\in R_1$ and $j\in T_1$ and 
		there is at least one oriented edge $(i, j)\in E$ with $i\in T_1$ and $j\in R_1$. 
		This means that the subgraph of $G$ given by  $(R_1, E_1)$, where  $E_1=E\cap \{(i, j): i, j\in R_1\}$, is irreducible and agents in $R_1$ form an autonomous system to which we can apply Theorem \ref{convergence}.
		
		Reordering the indices in such a way that $R_1=\{1, 2,\ldots, N_0\}$, with $N_0<N$, and $T_1=\{N_0+1,\ldots, N_0+N_1\}$ with $N_0+N_1=N$, the matrix $\mathbf{A}$ takes the following form 
		$$ \mathbf{A}=\left(\begin{array}{llcc}
			\mathbf{A}_0 &  \mathbf{0} \\
			\mathbf{B} & \mathbf{A}_1
		\end{array}\right),
		$$ 
		where $\mathbf{A}_0$ (respectively, $\mathbf{A}_1$) describes the  interaction weights among agents of the class $R_1$ (respectively, $T_1$) and the $N_1\times N_0$ matrix $\mathbf{B}$ has at least one positive entry and describes how agents in $T_1$ are influenced by agents in $R_1$. For a clearer use of this decomposition, we denote by $\mathbf{Z}_n= (Z_n(i))_{i=1}^{N_0}\in \R^{N_0}$ the state components in $R_1$, and by $\mathbf{U}_n=(U_n(i))_{i={N_0+1}}^{N}\in \R^{N_1}$ the state components in $T_1$, so that the equations for the dynamics take the following form
		\begin{align}
			&Z_{n+1}(i)=(1-r_n)Z_n(i) + r_n Y_{n+1}(i), &1\leq i\leq  N_0; \label{mainEQN1}\\
			&P\left(Y_{n+1}(i)=1|\mathcal{F}_n\right) = \alpha_{ii}Z_n(i)+ \sum_{j\neq i, j\in R_1}\alpha_{ij} \left(1-Z_n(j)\right), &1\leq i \leq N_0;\nonumber\\
			&U_{n+1}(i)=(1-r_n)U_n(i) + r_n Y_{n+1}(i), &N_0 < i \leq  N; \label{mainEQN2}\\
			&\begin{aligned}P\left(Y_{n+1}(i)=1|\mathcal{F}_n\right) = \alpha_{ii}U_n(i)&+ \sum_{j\neq i, j\in T_1}\alpha_{ij} \left(1-U_n(j)\right)\\
				&\qquad+\sum_{ j\in R_1}\alpha_{ij} \left(1-Z_n(j)\right),\end{aligned}  &N_0<i\leq  N.\nonumber
		\end{align}                                               
		
		Equations \eqref{mainEQN1} and \eqref{mainEQN2} can be written in vector form as  
		
		\begin{align}\label{RED}
			\mathbf{Z}_{n+1}&=\mathbf{Z}_n+ r_n (\mathbf{K}_0\mathbf{Z}_n+\mathbf{c}^{(0)})+ r_n \mathbf{\Delta M}^{(0)}_{n+1};\\
			\mathbf{U}_{n+1}&=\mathbf{U}_n +r_n (\mathbf{K}_1\mathbf{U}_n+\mathbf{c}^{(1)}_n)+ r_n\mathbf{\Delta M}^{(1)}_{n+1},
		\end{align}
		where $\mathbf{K}_i=2 \mathrm{diag}(\mathbf{A}_i)-{\mathrm{\mathbf{I}}}_{N_i}-\mathbf{A}_i,\ i=0,1$, with  $\mathrm{\mathbf{I}}_{N_i}$ denoting the identity matrix of dimension $N_i$, $\mathbf{c}^{(0)}=(1-\alpha_{ii})_{i=1}^{N_0}$, and $\mathbf{c}_n^{(1)}=\mathbf{c}^{(1)}-\mathbf{B}\mathbf{Z}_n$, with $\mathbf{c}^{(1)}=(1-\alpha_{ii})_{i=N_0+1}^N$.\\
		Equation \eqref{mainEQN1} is of the form \eqref{equation1}, for the interaction given by the restriction of $G$ to the vertices in $R_1$, which is irreducible, so by Theorem \ref{convergence} $\mathbf{Z}_n$ converges to a possibly random limit. Equation \eqref{mainEQN2} contains the sequence $\mathbf{Z}_n$ as a {\em forcing input}. For this reason we now devote our attention to a modification of \eqref{equation1} obtained by adding a forcing input.
		
		\subsection{The case with a forcing input}\label{sec:competitiveforce}
		In this section we assume that the interaction matrix $\mathbf{A}$ is substochastic with $\sum_{i=1}^N\alpha_{ij}=:\alpha_{i}\leq 1$ for every $i\in \{1,\ldots, N\}$ and  at least one $h\in \{1,\ldots, N\}$ such that $\alpha_h<1$. Let us consider a possibly random sequence $(\mathbf{q}_n)_n$ in  $[0,1]^N$ and the dynamics
		\begin{align}\label{equation3}
			&Z_{n+1}(i)=(1-r_n)Z_n(i) + r_n Y_{n+1}(i) ,\qquad\qquad\qquad\qquad\qquad\qquad\qquad\quad i=1,\ldots, N,\\
			&P\left(Y_{n+1}(i)=1|\mathcal{F}_n\right) = \alpha_{ii}Z_n(i)+ \sum_{j\neq i}\alpha_{ij} \left(1-Z_n(j)\right)+(1-\alpha_i)q_{n}(i), \ i=1,\ldots, N,\nonumber
		\end{align}  
		where $q_{n}(i)$ is the $i$-th component of $\mathbf{q}_n$.

		Models of this type have been  studied for cooperative systems in the case of $q_{n}(i) \equiv q$ (see, e.g. \cite{aletti2017synchronization}). Our results, which include random non-constant sequences, can be extended to cooperative systems as we point out in the Remark \ref{THMQcooperative} at the end of the section, thus extending the results in \cite{aletti2017synchronization}.\\

		Equation \eqref{equation3} in vector form is given by 
		\begin{equation}\label{equation5}
			\mathbf{Z}_{n+1}=\mathbf{Z}_n+ r_n (\mathbf{K}\mathbf{Z}_n+\mathbf{c}_n)+ r_n \mathbf{\Delta M}_{n+1},
		\end{equation}
		where $\mathbf{K}=2\mathrm{diag}(\mathbf{A})-{\mathrm{\mathbf{I}}}_{N}-\mathbf{A}=(k_{ij})_{i,j=1}^N$  
		and $\mathbf{c}_n=(c_n(i))_{i=1}^N$ with $c_n(i)=\alpha_i-\alpha_{ii}+(1-\alpha_i)q_{n}(i)$. 
		
		\begin{theorem} \label{forcing}
			The matrix $\mathbf{K}$ is invertible. Moreover, suppose there exists a random vector $\mathbf{q}=(q(i))_{i=1}^N \in [0,1]^N$ such that
			\[
			\lim_{n \rightarrow +\infty} q_n(i) = q(i)\quad  \text{a.s.},
			\]
			for all $i=1,2,\ldots,n$. Then, as $n \rightarrow +\infty$, $\mathbf{Z}_n$ converges a.s. to $-\mathbf{K}^{-1} \mathbf{c}$, where  $\mathbf{c}=(c(i))_{i=1}^N$ is given by
			\[
			c(i) = \lim_{n \rightarrow +\infty} c_n(i) = \alpha_i-\alpha_{ii}+(1-\alpha_i)q(i).
			\]
			
		\end{theorem}
		
		\begin{proof}
			The invertibility of $\mathbf{K}$ follows from Theorem \ref{gervsgorin} (ii). Theorem \ref{benaim} applied to $F(x) = \mathbf{K}(x+\mathbf{K}^{-1}\mathbf{c})$ yields the desired result in the case $\mathbf{c}_n \equiv \mathbf{c}$. Otherwise, just observe that \eqref{equation5} can be rewritten as
			\[
			\mathbf{Z}_{n+1}=\mathbf{Z}_n+ r_n (F(\mathbf{Z}_n) + \mathbf{b}_n)+r_n \mathbf{\Delta M}_{n+1},
			\]
			with $\mathbf{b}_n := \mathbf{c}_n - \mathbf{c} \xrightarrow{\text { $n \rightarrow +\infty$ }} 0$. As observed in Remark 4.5 of \cite{Benaim99}, Theorem \ref{benaim} extends readily to this case.
			
		\end{proof}
		
		\begin{remark}\label{THMQcooperative}
			The results above extend to the cooperative case, i.e. when in \eqref{equation3} the term $\sum_{j\neq i}\alpha_{ij} \left(1-Z_n(j)\right)$ is replaced by $\sum_{j\neq i}\alpha_{ij} Z_n(j)$. 
			In this case the dynanics can still be written in the form \eqref{equation5}, with $\mathbf{K}=\mathbf{A}-\mathbf{I}_N$ and  $\mathbf{c}_n=\big((1-\alpha_i)q_n(i)\big)_{i=1}^N$. 
		\end{remark}

		\subsection{Reducible graphs}
		
		We now generalize the representation obtained in \eqref{RED}.
		In order to study the general case we construct a hierarchical representation of the graph $G$. 
		We call \emph{class of level 0} a \emph{closed} class $R$ of communicating agents (i.e. a subset of $V$ of communicating agents such that $\alpha_{i,j}=0$ for any $i\in R$ and $j\in R^C$).\\
		Any finite graph $G$ contains at least one class of level 0.\\ 
		
		We denote by $R_1,\ldots, R_{k_0}$, with $k_0\in \{1,\ldots, N\}$, the classes of level 0 in $G$ and we set $\mathbf{R}=\bigcup_{\ell=0}^{k_0}R_\ell$.\\

		A \emph{class of level 1} is a set $T\subset V$ of communicating agents such that $\alpha_{i,j}=0$, for any $i\in T$ and $j\in T^C\cap \mathbf{R}^C$, and there exist at least one $i\in T$ and $j\in \mathbf{R}$ such that $\alpha_{i,j}>0$.  
		
		We denote by $T_{1,1}, T_{1,2},\ldots T_{1, k_1}$ with $k_1\in \{1, \ldots, N-|\mathbf{R}|\}\ $ (where $|\mathbf{R}|$ is the cardinality of $\mathbf{R}$) the classes of level 1 (which are necessarily disjoint) and we set $\mathbf{T_1}=\bigcup_{\ell=1}^{k_1}T_{1,\ell}$. 
		
		For $2\leq m\leq N-1$, a \emph{class of level} $m$ is a set $T\subset V$ of communicating agents such that there exists at least one $i\in T$ and $j\in \mathbf{T_{m-1}}$ such that $\alpha_{i,j}>0$ and $\alpha_{i,j}=0$, for any $i\in T$ and $j\in T^C\cap \big[\big(\bigcup_{s=1}^{m-1} \mathbf{T_s}\big)^C \cap  \mathbf{R}^C\big]$. 
		
		We denote by $T_{m, 1}, T_{m, 2}, \ldots, T_{m, k_m}$ with $k_m\in \{1, \ldots N-|R\cup T_{\mathbf{1}}\cup\ldots\cup T_{\mathbf{m-1}}|\}$
		the classes of level $m$, which are clearly disjoint. 
		
		We have obviously $\mathbf{T_{m}}\cap \mathbf{R}=\emptyset$ for every $m$ and $\mathbf{T_s}\cap \mathbf{T_m}=\emptyset$, for every $s, m$.\\
		
		The idea is that each class of level 0 is closed, each class of level $1$ has at least one edge connecting it to level 0 and no edges connecting it to other classes of level greater than or equal to $1$ and, for $m\geq 1$, each class of level $m$ has at least one edge connecting it to level $m-1$ and no edges connecting it to classes of level greater than or equal to $m$. This means that agents of level $m$ interact only with agents in the same class or in classes of lower level.

		Let $M$ be the number of levels and let $N_{m,\ell}=|T_{m,\ell}|\ $, $m=0,\ldots , M,\ \ \ell=1,\ldots k_m$ (where we have used the convention $R_\ell=T_{0, \ell}$). We have $\sum_{m=1}^M \sum_{\ell=1}^{k_m} N_{m,\ell}=N$ and we can assume that agents are ordered according to levels and classes in such a way that \begin{multline*} 
			\{1,\ldots, N\}=
			\{1,\ldots, N_{0,1},\  N_{0,1}+1, \ldots, N_{0,1}+N_{0,2},\  N_{0, 1}+N_{0,2}+1, \ldots ,\sum_{\ell=1}^{k_{0}}N_{0,\ell},\ \ldots  \\
			\ldots\ ,\  \sum_{m=1}^{M-1}\sum_{\ell=1}^{k_{m}}N_{m,\ell}+1
			,\ldots,\sum_{m=0}^{M}\sum_{\ell=1}^{k_{m}} N_{m,\ell}\}.
		\end{multline*}  In this case the interaction matrix can be written as follows:
		
		$$ \mathbf{A}=\left(\begin{matrix}
			\mathbf{A_0} &  0  & 0 &\cdots & 0 & 0\\
			\mathbf{B_{1,0}} & \mathbf{A_1}& 0 &\cdots & 0 & 0\\
			\mathbf{B_{2,0}} & \mathbf{B_{2,1}} & \mathbf{A_2}&\cdots & 0 & 0\\
			\cdots & \cdots & \cdots & \cdots & \cdots & \cdots\\
			\mathbf{B_{M,0}} & \mathbf{B_{M,1}} & \mathbf{B_{M,2}}&\cdots &\mathbf{B_{M,M-1}} &  \mathbf{A_M}
		\end{matrix}\right).
		$$                                                              
		In the above formula 
		$$ \mathbf{A_0}=\left(\begin{matrix}
			\mathbf{A}_{0,1} &  0  & \cdots &0 \\
			0 & \mathbf{A}_{0,2}  &  \cdots &0\\
			\cdots & \cdots & \cdots & \cdots \\
			0 & 0 & \cdots & \mathbf{A}_{0,k_0},
		\end{matrix}\right)
		$$  
		and, in general, 
		$$ \mathbf{A_{m}}=\left(\begin{matrix}
			\mathbf{A}_{m,1} &  0  & \cdots &0 \\
			0 & \mathbf{A}_{m,2}  &  \cdots &0\\
			\cdots & \cdots & \cdots & \cdots \\
			0 & 0 & \cdots & \mathbf{A}_{m,k_m}
		\end{matrix}\right),
		$$      
		where $\mathbf{A}_{0,\ell}=(\alpha_{i,j})_{i, j \in R_{\ell}}$ and $  \mathbf{A}_{m, \ell}=(\alpha_{i,j})_{i, j \in T_{m, \ell}}$, for $1\leq m\leq M$.
		The matrix $\mathbf{A_m}$ describes self interactions in classes of level $m$. More precisely, 
		for $m=0,\ldots, M$ and $\ell=1,\ldots, k_m$, the matrix $\mathbf{A}_{m,\ell}$ represents how agents in class $\ell$ of level $m$ influence each other.\\

		The matrix $\mathbf{B_{1, 0}}$ describes how classes of level 1 are influenced by classes of level 0. It is given by\\
		
		$$ \mathbf{B_{1, 0}}=\left(\begin{matrix}
			\mathbf{B}_{(1,1)}^{(0,1)} &   \mathbf{B}_{(1,1)}^{(0,2)}  & \cdots &\mathbf{B}_{(1,1)}^{(0, k_0)} \vspace{0.2cm} \\
			\mathbf{B}_{(1, 2)}^{(0,1)}  & \mathbf{B}_{(1,2)}^{(0,2)}  &  \cdots & \mathbf{B}_{(1,2)}^{(0, k_0)}  \vspace{0.2cm}\\
			\cdots & \cdots & \cdots & \cdots \\
			\cdots & \cdots & \cdots & \cdots  \vspace{0.2cm} \\
			\mathbf{B}_{(1, k_1)}^{(0,1)}  & \mathbf{B}_{(1,  k_1)}^{(0, 2)}  & \cdots & \mathbf{B}_{(1,k_1)}^{(0,k_0)} 
		\end{matrix}\right),
		$$      
		and in general, for $0\leq t<m$, $\mathbf{B_{m,t}}$ describes how classes of level $m$ are influenced by classes of level $t<m$. It is given by 
		$$ \mathbf{B_{m, t}}=\left(\begin{matrix}
			\mathbf{B}_{(m,1)}^{(t,1)} &   \mathbf{B}_{(m ,1)}^{(t,2)}  & \cdots &\mathbf{B}_{(m,1)}^{(t, k_t)} \vspace{0.2cm} \\
			\mathbf{B}_{(m, 2)}^{(t,1)}  & \mathbf{B}_{(m,2)}^{(t,2)}  &  \cdots & \mathbf{B}_{(m,2)}^{(t, k_t)}  \vspace{0.2cm}\\
			\cdots & \cdots & \cdots & \cdots \\
			\cdots & \cdots & \cdots & \cdots  \vspace{0.2cm} \\
			\mathbf{B}_{(m, k_m)}^{(t,1)}  & \mathbf{B}_{(m,  k_m)}^{(t, 2)}  & \cdots & \mathbf{B}_{(m,k_m)}^{(t, k_t)} 
		\end{matrix}\right),
		$$     
		
		where, for  $m=1,\ldots, M,\ \  t=0,\ldots , m-1,\   h=1,\ldots k_m,\ s=1,\ldots, k_t$, we have
		\begin{eqnarray*}
			\mathbf{B}_{(m , h)}^{(0, s)} &=& (\alpha_{i,j})_{i\in T_{m, h}, \  j\in R_s},\ \ \ \ t=0,\\ 
			\mathbf{B}_{(m , h)}^{(t, s)} &=& (\alpha_{i,j})_{i\in T_{m, h},\  j\in T_{t, s}},\ \ \  t>0. 
		\end{eqnarray*}
		More precisely, $\mathbf{B}_{(m , h)}^{(t, s)}$ represents how agents in class $h$ of level $m$ are influenced by agents in class $s$ of level $t$.\\
		Recall that, by construction, dynamics of agents in classes of level from 0 to $m$ are autonomous with respect to dynamics of classes with level higher than $m$. \\
		Let us denote by 
		$\mathbf{Z}^{(\ell)}_n\in \R^{N_{0, \ell}},\ \ \ell=1,\ldots k$ the ''class $\ell$ variable'' of level 0 (i.e. the vector whose components are the states of agents belonging to class $\ell$ of level 0), and by \\
		$\mathbf{U}^{(m, h)}_n\in \R^{N_{m,h}},\ \ m=1,2,\ldots M,\ h=1,\ldots, k_m$ the ''class $h$ variable'' of level $m$ (i.e. the vector of states of agents belonging to class $h$ of level $m$.\\
		The vector of states is thus given by $$\mathbf{\bar{Z}}_n=\big( (\mathbf{Z}^{(\ell)}_n)_{1\leq\ell\leq k_0}; (\mathbf{U}^{(1, h)}_n)_{1\leq h\leq k_1}; \ldots ;(\mathbf{U}^{(M, h)}_n)_{1\leq h\leq k_M}\big).$$
		
		Assume for the moment that all the classes are competitive, in such a way that the dynamics can be described by equations of the type \eqref{mainEQN1} and \eqref{mainEQN2}, where $R_1$ contains all the agents in classes of level 0 and  $T_1=R_1^C$. 
		
		Letting, for $m=0,1,\ldots, M$ and $\ell=1,\ldots, k_m$, 
		$\ \mathbf{K}_{(m, \ell)}:=2 \mathrm{diag}(\mathbf{A}_{(m,\ell)})-{\mathrm{\mathbf{I}}}_{N_{m,\ell}}-\mathbf{A}_{(m,\ell)}$  with  $\mathrm{\mathbf{I}}_{N_{m, \ell}}$ denoting the identity matrix of dimension $N_{m,\ell}$, the system dynamics is described by the following equations
		
		\begin{align}
			&\mathbf{Z}_{n+1}^{(\ell)}=\mathbf{Z}^{(\ell)}_n+ r_n (\mathbf{K}_{(0,\ell)}\mathbf{Z}_n^{(\ell)}+\mathbf{c}^{(0, \ell)})+ r_n \mathbf{\Delta M}^{(0,\ell)}_{n+1}, &1\leq \ell\leq k, \label{EQNz}\\
			&\ \mathbf{c}^{(0,\ell)}=(1-\alpha_{ii})_{i\in R_\ell},\nonumber\\
			& \nonumber\\
			&\mathbf{U}^{(1,h)}_{n+1}=\mathbf{U}^{(1,h)}_n +r_n\big( \mathbf{K}_{(1,h)}\mathbf{U}^{(1,h)}_n+\mathbf{c}^{(1,h)}_n\big)+ r_n\mathbf{\Delta M}^{(1,h)}_{n+1},&    1\leq h\leq k_1,\label{EQNm1}\\
			&\ \mathbf{c}^{(1, h)}_n=(1-\alpha_{ii})_{i\in T_{1, h}}- \sum_{\ell=1}^{k_0} \mathbf{B}_{(1, h)}^{(0, \ell)}\mathbf{Z}^{(\ell)}_n,\nonumber
		\end{align}
		
		and, for  $m=1,\ldots ,M-1$ and $h=1,\ldots, k_{m+1}$,  
		
		\begin{align}\label{EQNm}
			&\mathbf{U}^{(m+1, h)}_{n+1}=\mathbf{U}^{(m+1, h)}_n +r_n \big(\mathbf{K}_{(m+1, h)} \mathbf{U}^{(m+1, h)}_n + \mathbf{c}^{(m+1, h)}_n\big)+ r_n \mathbf{\Delta M}^{{(m+1, h)}}_{n+1},\\
			&\ \mathbf{c}^{(m+1, h)}_n=\mathbf{c}^{(m+1, h)}- \sum_{\ell=1}^{k_0} \mathbf{B}_{(m+1, h)}^{(0, \ell)}\mathbf{Z}^{(\ell)}_n -\sum_{t=1}^m \sum_{s=1}^{k_t} \mathbf{B}_{(m+1, h)}^{(t, s)} \mathbf{U}^{(t, s)}_n \nonumber,
		\end{align} 
		where in the above formulas, for any fixed $m\in \{0,\ldots, M\}$ and $h\in\{1,\ldots k_m\}$, the sequence $\ (\mathbf{\Delta M}^{(m,\ell)}_{n})_{n\geq 1}$
		is a martingale difference and $\mathbf{c}^{(m+1, h)}=\big(1-\alpha_{ii})_{i\in T_{m+1, h}}$.\\
		
		Now, observe that we can repeat the arguments above if $R_{\ell}$ or $T_{m+1, h},\ m\geq 0$ is a class of cooperative agents: in this case in equations \eqref{mainEQN1} (respectively, \eqref{mainEQN2}), for $i\in R_{\ell}$ (respectively, $i\in T_{m+1, h}$) we have to replace $1-Z_n(j)$ with $Z_n(j)$ (respectively $1-U_n(j)$ with $U_n(j)$). In this cases $(\mathbf{Z}^{(\ell)}_n)_n$ satisfies the equation 
		\begin{equation}\label{EQNCooperativeZ}
			\mathbf{Z}^{(\ell)}_{n+1}=\mathbf{Z}^{(\ell)}_n +r_n \mathbf{K}_{(0,\ell)} \mathbf{Z}^{(0,\ell)}_n+ r_n \mathbf{\Delta M}^{{(0, \ell)}}_{n+1},
		\end{equation} 
		with $\mathbf{K}_{(0,\ell)}=\mathbf{A}_{(0,\ell)}-\mathbf{I}_{N_{0,\ell}}$, while  $(\mathbf{U}^{(m+1, h)}_n)_n$ satisfies  
		
		\begin{align}
			& \mathbf{U}^{(1,h)}_{n+1}=\mathbf{U}^{(1,h)}_n +r_n\big( \mathbf{K}_{(1,h)}\mathbf{U}^{(1,h)}_n+\mathbf{c}^{(1,h)}_n\big)+ r_n\mathbf{\Delta M}^{(1,h)}_{n+1},\label{EQNmcooperative1}\\
			&\ \mathbf{c}^{(1, h)}_n= \sum_{\ell=1}^{k_0} \mathbf{B}_{(1, h)}^{(0,\ell)}\mathbf{Z}^{(\ell)}_n,\nonumber\\
			&\nonumber\\ 
			&\mathbf{U}^{(m+1, h)}_{n+1}=\mathbf{U}^{(m+1, h)}_n +r_n \big(\mathbf{K}_{(m+1, h)} \mathbf{U}^{(m+1, h)}_n + \mathbf{c}^{(m+1, h)}_n\big)+ r_n \mathbf{\Delta M}^{{(m+1, h)}}_{n+1}, \label{EQNmcooperative}\\
			&\ \mathbf{c}^{(m+1, h)}_n=\sum_{\ell=1}^{k_0} \mathbf{B}_{(m+1, h)}^{(0, \ell)}\mathbf{Z}^{(\ell)}_n +\sum_{t=1}^m \sum_{s=1}^{k_t} \mathbf{B}_{(m+1, h)}^{(t, s)} \mathbf{U}^{(t, s)}_n, \nonumber
		\end{align} 
		with $\mathbf{K}_{(m+1, h)}=\mathbf{A}_{m+1, h}-\mathbf{I}_{N_{m+1,h}}$, for $m\geq 0$. \\
		
		We then suppose that classes can be cooperative or competitive and using the hierarchical representation of $G$, thanks to equations \eqref{EQNz}, \eqref{EQNm} and \eqref{EQNCooperativeZ}, \eqref{EQNmcooperative}, we can easily prove the following result:

		\begin{theorem}
			Assume that in the hierarchical representation  of the graph $G=(V,E)$ associated to the system, agents in a given class have all the same attitude (i.e., they are all cooperative or all competitive). Let $\mathbf{\bar{Z}}_n\in \R^N$ the vector of agents' states at step $n$. Then $\mathbf{\bar{Z}}_n\longrightarrow \mathbf{\bar{Z}}_\infty$ a.s., where $$\mathbf{\bar{Z}}_\infty=\big( (\mathbf{Z}^{(\ell)}_\infty)_{1\leq\ell\leq k_0}; (\mathbf{U}^{(1, h)}_\infty)_{1\leq h\leq k_1}; \ldots ;(\mathbf{U}^{(M, h)}_\infty)_{1\leq h\leq k_M}\big).$$ In the above formula, for $\ell\in \{1,\ldots, k_0\}$, if the class $R_\ell$ is cooperative, $\mathbf{Z}^{(\ell)}_\infty= (Z_\infty^{(\ell)},\ldots,$  $Z^{(\ell)}_\infty)\in \R^{N_{0,\ell}}$, where $Z^{(\ell)}_\infty$ is a random variable with $P(Z^{(\ell)}=x)<1,\  \forall x\in [0,1]$; if the class $R_\ell$ is competitive and the subgraph $G_\ell=(R_{\ell},\  E\cap R_\ell^2)$ of $G$ associated to $(\alpha_{i,j})_{i,j\in R_\ell}$ is bypartite,
			$\mathbf{Z}^{\ell}_\infty=(Z^{(\ell)}_\infty,\ldots , Z^{(\ell)}_\infty,\  1-Z^{(\ell)}_\infty,\ldots, 1-Z^{(\ell)}_\infty)\in \R^{N_{0,\ell}}$, 
			where $Z^{(\ell)}_\infty$ is a random variable with $P(Z^{(\ell)}=x)<1,\  \forall x\in [0,1]$; 
			if the class $R_\ell$ is competitive and the graph $G_\ell$ is not bypartite, 
			$\mathbf{Z}^{(\ell)}_\infty=(\frac{1}{2},\ldots, \frac{1}{2})\in \R^{N_{0,\ell}}$;  for $m\in \{1,\ldots M\}$ and $h\in \{1,\ldots, k_m\}$ the random variable $\mathbf{U}^{(m,h)}_\infty$ depends on the variables of $\mathbf{Z}^{(\ell)}_\infty,\ \ell=1,\ldots, k_0$, according to the following recursive formula:   
			\begin{eqnarray*}
				& &\mathbf{U}^{(m, h)}_{\infty}=-\mathbf{K}_{(m, h)}^{-1} \Big(\mathbf{c}^{(m, h)}+ \sum_{\ell=1}^{k_0} \widetilde{\mathbf{B}}_{(m, h)}^{(0, \ell)}\mathbf{Z}^{(\ell)}_\infty +\sum_{t=1}^{m-1} \sum_{s=1}^{k_t} \widetilde{\mathbf{B}}_{(m, h)}^{(t, s)} \mathbf{U}^{(t, s)}_\infty\Big),
			\end{eqnarray*}
			where, for $T_{m,h}$ competitive or cooperative, we have respectively, for $ i\in\{0,\ldots m-1\}$ and $j\in \{1,\ldots, k_i\}$, 
			\begin{equation*}
				\begin{array}{lllccc}
					&\mathbf{K}_{(m,h)}= 2\mathrm{diag}(\mathbf{A}_{m,h})+\mathrm{I}_{N_{m,h}}-\mathbf{A}_{m,h}, &\mathbf{c}^{(m, h)}=(1-\alpha_{ii})_{i\in T_{m, h}}, & \widetilde{\mathbf{B}}_{(m,h)}^{(i,j)}=-\mathbf{B}_{(m,h)}^{(i,j)},\\  
					& \\
					&\mathbf{K}_{(m,h)}=\mathbf{A}_{m,h}-\mathrm{I}_{N_{m,h}},  &\mathbf{c}^{(m, h)}=\mathbf{0}, &\widetilde{\mathbf{B}}_{(m,h)}^{(i,j)}=\mathbf{B}_{(m,h)}^{(i,j)}.\\
				\end{array}
			\end{equation*}
		\end{theorem}
		\begin{proof}
			We proceed by induction on the level $m$. We first observe that the asymptotic behavior of classes in the same level can be studied separately, since they interact only with classes of lower levels. 
			Any class of level $0$ satisfies equation \eqref{EQNz} if it is competitive, and equation \eqref{EQNCooperativeZ} if it is cooperative, so its asymptotic behavior can be derived by the results in the first section. By Theorems \ref{convergence} and \ref{convergenceCoop}, $\mathbf{Z}^{(\ell)}_n\longrightarrow \mathbf{Z}^{(\ell)}_\infty$, for every $\ell=1,\ldots k_0$. For $m=1$ and $h=1,\ldots, k_1$, the sequence $(\mathbf{U}^{(1,h)}_n)_n$ satisfies equation \eqref{EQNm1} if the corresponding class $T_{1,h}$ is competitive, and equation \eqref{EQNmcooperative1} if the class $T_{1,h}$ is cooperative. In both cases we have $\mathbf{c}^{(1,h)}_n\longrightarrow \mathbf{Q}^{(1,h)}$ a.s. with $\mathbf{Q}^{(1,h)}=(1-\alpha_{ii})_{i\in T_{1,h}}-\sum_{\ell=1}^{k_0}\mathbf{B}_{(1, h)}^{(0, \ell)}\mathbf{Z}^{(\ell)}_\infty$ if $T_{1, h}$ is competitive, and  $\mathbf{Q}^{(1,h)}=\sum_{\ell=1}^{k_0} \mathbf{B}_{(1, h)}^{(0, \ell)}\mathbf{Z}^{(\ell)}_\infty$ if $T_{1, h}$ is coperative. 
			Now, equations \eqref{EQNm1} and \eqref{EQNmcooperative1} have the same form as the ones considered in Section \ref{sec:irred}, and satisfy the same hypotheses. Therefore, by Theorem \ref{forcing} and Remark \ref{THMQcooperative}, we get that, for every $h=1,\ldots k_1$,  $\mathbf{U}^{(1,h)}_n\longrightarrow \mathbf{U}^{(1,h)}_\infty$, where $\mathbf{U}^{(1,h)}_\infty=-\mathbf{K}_{(1,h)}^{-1}\mathbf{Q}^{(1, h)}$.\\
			By induction, for $m>1$, we can suppose that  $\mathbf{Z}^{(\ell)}_n\longrightarrow \mathbf{Z}^{(\ell)}_\infty$ a.s. , for every $\ell=1,\ldots, k_0$, and that $\mathbf{U}^{(t,s)}_n\longrightarrow \mathbf{U}^{(t,s)}_\infty$ a.s., for every $1\leq t\leq m$, $\ 1\leq s\leq k_t$. Then, arguing as above, using equations \eqref{EQNm}, \eqref{EQNmcooperative} and the results of Section \ref{sec:irred}, we get the conclusion.  
		\end{proof}
		\begin{remark}
			The above result clearly holds also if among the classes of level 0 we include \emph{stubborn agents}, i.e. agents with a (possibly random) state $q$ that does not change in time. 
		\end{remark}

		\bibliographystyle{abbrv}
		\bibliography{references}

		
	\end{document}